\documentclass[11pt]{amsart}
\usepackage{amsmath,amsfonts,amssymb,amsthm}
\def\w{\omega}
\def\z{\mathfrak{z}}

\def\g{\mathfrak{g}}
\def\go{\mathfrak{g}_0}
\def\gcc{\mathfrak{g}_0'}
\def\gccc{(\mathfrak{g}_0')^\perp}
\def\h{\mathfrak{h}}

\def\o{\theta}
\def\u{\mu}
\def\l{\lambda}
\def\ll{\lambda'}
\def\L{\Lambda}

\def\b{\beta}
\def\gd{\mathfrak{g}^*}

\def\gl{\g_\l}
\def\gm{\g_\mu}
\def\gll{\g_{\l'}}

\def\r{\mathbb{R}}

\def\aff{\mathfrak{aff}}

\def\^{\wedge}

\def\c{\mathbb{C}}

\newcommand{\nc}{\newcommand}
\nc{\ad}{\operatorname{ad}}
\nc{\tr}{\operatorname{tr}}
\nc{\I}{\operatorname{Id}}
\nc{\alt}{\raise1pt\hbox{$\bigwedge$}}
\nc{\pint}{\langle \cdotp,\cdotp \rangle }
\nc{\la}{\langle}
\nc{\ra}{\rangle}
\nc{\n}{\noindent}

\theoremstyle{plain}
\newtheorem{teo}{\bf Theorem}[section]
\newtheorem{cor}[teo]{\bf Corollary}
\newtheorem{prop}[teo]{\bf Proposition}
\newtheorem{lema}[teo]{\bf Lemma}

\theoremstyle{definition}

\newtheorem{ejemplo}[teo]{\bf Example}

\theoremstyle{remark}
\newtheorem{nota}[teo]{\bf Remark}
\newtheorem*{rem}{\bf Remark}
\newtheorem*{rems}{\bf Remarks}

\newcommand{\ri}{{\rm (i)}}
\newcommand{\rii}{{\rm (ii)}}
\newcommand{\riii}{{\rm (iii)}}
\newcommand{\riv}{{\rm (iv)}}

\title[LCK structures on Lie groups]{Locally conformally K\"ahler structures
on unimodular Lie groups}

\author{A. Andrada}
\email{andrada@famaf.unc.edu.ar}
\author{M. Origlia}
\email{origlia@famaf.unc.edu.ar}

\date{}
\address{FaMAF-CIEM, Universidad Nacional de C\'{o}rdoba, Ciudad Universitaria,
5000 C\'{o}rdoba, Argentina}

\thanks{The authors were partially supported by CONICET, ANPCyT and SECyT-UNC (Argentina).}

\begin{document}

\begin{abstract}
We study left-invariant locally conformally K\"ahler structures on Lie groups, or equivalently, on
Lie algebras. We give some properties of these structures in general, and then we consider the
special cases when its complex structure is bi-invariant or abelian. In the former case, we show
that no such Lie algebra is unimodular, while in the latter, we prove that if the Lie algebra is
unimodular, then it is isomorphic to the product of $\r$ and a Heisenberg Lie algebra.
\end{abstract}

\maketitle

\section{Introduction}

\
Let $(M,J,g)$ be a $2n$-dimensional Hermitian manifold and let $\omega$ be its
fundamental $2$-form, that is,
$\omega(X,Y)=g(JX,Y)$ for any $X,Y$ vector fields on $M$. The manifold $(M,J,g)$
is called {\it locally conformally K\"ahler} (or
l.c.K., for short) if $g$ can be rescaled locally, in a neighborhood of any
point in $M$, so as to be K\"ahler, or equivalently,
if there exists a closed $1$-form $\theta$ such that
\[d\omega=\theta\wedge\omega.\]
This $1$-form $\theta$ is called the {\it Lee form}. This notion was introduced
by P. Libermann
\cite{L1} in $1954$, but the geometry of these manifolds was not developed until
the 70's, with the work of  I. Vaisman.
These manifolds are a natural generalization of the class of K\"ahler manifolds,
and they have been much studied by many authors
(see for instance \cite{DO, S1, V2}). According to \cite{GH}, a locally
conformally K\"ahler manifold is in the class $\mathcal
W_4$ of the Gray-Hervella classification of almost Hermitian manifolds. An important class of
l.c.K. metrics is given by those whose Lee form is parallel with respect to the Levi-Civita
connection. These l.c.K. structures are called \textit{Vaisman}, and their existence imposes
topological and cohomological restrictions on the underlying Hermitian manifold (see for instance
\cite{V2}). 

We will consider locally conformally K\"ahler structures on solvmanifolds, that
is, compact quotients
$\Gamma\backslash G$ where $G$ is a simply connected solvable Lie group and
$\Gamma$ is a lattice in $G$, which are induced by
left-invariant locally conformally K\"ahler structures on the Lie group. These
structures have been the subject of study in
several recent papers. For instance, it was shown in \cite{S} that if a
non-toral nilmanifold  admits an invariant locally
conformally K\"ahler structure, then it is a quotient of $\r \times H_{2n+1}$,
where $H_{2n+1}$ is the $(2n +1)$-dimensional
Heisenberg Lie group. In \cite{S1} it was proved that any invariant locally
conformally K\"ahler structure on a solvmanifold such that $\w=-\o\wedge J\o +d(J\o)$ is in fact
Vaisman. According to \cite{OV1} this condition is related to the
existence of a potential for the l.c.K. metric. In \cite{K} it is proved the
non-existence of Vaisman metrics on solvmanifolds satisfying certain
cohomological conditions. 

In this article we study l.c.K. structures on Lie algebras with two special
kinds of complex structures. First, we take into account {\it bi-invariant}
complex structures on Lie algebras, i.e., an endomorphism $J$ of a Lie algebra
$\g$ that satisfies
\[ J^2=-I, \hspace{2cm} J[X,Y]=[X,JY] \, \, \text{for all} \, \,  X,Y\in \g.\]
This condition holds if and only if both left- and right-translations on the
corresponding simply connected (real) Lie group are holomorhic, or equivalently,
this Lie group is in fact a complex Lie group.

The other special kind of complex structures that we consider is given by the so
called {\it abelian} complex structures. Recall that an
abelian complex structure on a Lie algebra $\g$ is an endomorphism $J$ of $\g$
that satisfies
\[ J^2=-I, \hspace{2cm} [JX,JY]=[X,Y] \, \, \text{for all} \, \,  X,Y\in \g,\]
or equivalently,  the $i$-eigenspace of $J$ in $\g^\mathbb C$ is an abelian
subalgebra of $\g^\mathbb C$.
There are well known obstructions for the existence of abelian
complex structures. For instance, if the Lie algebra $\g$ admits
such a structure, then $\g$ has abelian commutator (i.e., $\g$ is
two-step solvable), and the center of $\g$ is $J$-invariant, among
other properties (see Lemma \ref{prop}). These structures are very
important in several areas of geometry and they have been studied by
many authors recently (see for instance \cite{BDV, DF, mpps}).
 
The outline of this article is as follows. In Section $2$ we review some known results about
l.c.K.~manifolds and left-invariant complex structures on Lie groups. In Section $3$ we determine
some properties of Lie algebras endowed with an l.c.K.~or Vaisman structure. Next, in Section $4$ we
prove that there exists no unimodular Lie algebra $\g$ equipped with an l.c.K.
structure $(J,\pint)$ where $J$ is bi-invariant (Theorem \ref{J-bi}). Finally, in Section $5$ we
prove that if $(\g,J,\pint)$ is l.c.K. with an abelian complex structure $J$ and $\g$ is unimodular
then $\g \simeq \r\times\h_{2n+1}$, where $\mathfrak{h}_{2n+1}$ is the $(2n+1)$-dimensional
Heisenberg Lie algebra. Moreover, there is only one, up to equivalence, monoparametric family
$(J_0,\pint_\l),\, \l>0$, of l.c.K. structures on this Lie algebra, where the metrics $\pint_\l$
are pairwise non-isometric.

\

\section{Preliminaries}

\

\subsection{Locally conformally K\"ahler manifolds}

\

A Hermitian metric on an almost complex manifold $(M,J)$ is a
Riemannian metric $g$ such that $g(X,Y)=g(JX,JY)$ for any vector
fields $X,Y$ on $M$. In this case $(M,J,g)$ is called an almost
Hermitian manifold. When the almost complex structure $J$ is
integrable (i.e., $(M,J)$ is a complex manifold), then $(M,J,g)$ is
called a Hermitian manifold.

Given an almost Hermitian manifold $(M,J,g)$, the {\em fundamental 2-form} is defined  by
$\omega(X,Y)=g(JX,Y)$ for any vector fields $X,Y$ on $M$.

A {\em K\"ahler metric} on a complex manifold $(M,J)$ is a Hermitian metric $g$ whose fundamental
$2$-form $\omega$ is closed, that is, $d\omega=0$. Then $M$ is called a K\"ahler manifold.

\medskip

K\"ahler manifolds are by far the most important Hermitian
manifolds. Nevertheless, this condition might be very restrictive in
some cases, and therefore, weaker conditions are studied. One way to
do so is to consider Hermitian manifolds whose metric is locally
conformal to a K\"ahler metric.

The Hermitian manifold $(M,J,g)$ is  {\em locally conformally K\"ahler}
(l.c.K.) if there exists an open covering
$\{ U_i\}_{i\in I}$ of $M$ and a family $\{ f_i\}_{i\in I}$ of
$C^{\infty}$-functions, $f_i:U_i \to \r$, such that each local
metric
\begin{equation}\label{gi}
g_i=\exp(-f_i)\,g|_{U_i}
\end{equation}
is K\"ahler. Also $(M,J,g)$ is {\em globally conformally K\"ahler} (g.c.K.)
if there exists a $C^{\infty}$-function, $f:
M\to\r$, such that the metric $\exp(f)g$ is K\"ahler.


\medskip

We recall an important result which characterizes l.c.K. manifolds in terms
of its fundamental form (see \cite{DO} for a proof).

\begin{teo}[\cite{L}]\label{dw=tita*w}
The Hermitian manifold $(M,J,g)$ is l.c.K. if and only if there exists a
closed $1$-form $\theta$ globally defined on $M$ such
that
\begin{equation}\label{lck}
d\omega=\theta\wedge\omega.
\end{equation}
Moreover, $(M,J,g)$ is globally conformally K\"ahler if and only the $1$-form $\theta$ in
\eqref{lck} is exact.
\end{teo}

\medskip

\begin{rems}
(i) A simply connected l.c.K. manifold is g.c.K., in particular the
universal cover of a l.c.K. manifold is g.c.K.

\smallbreak\n(ii) An l.c.K. manifold $(M,J,g)$ is K\"ahler if
and only if $\theta=0$. Indeed, $\theta\wedge\omega=0$ and $\omega$
non-degenerate imply $\theta=0$.

\smallbreak\n(iii) It is known that if $(M,J,g)$ is a Hermitian manifold with
$\dim M\ge 6$ such that \eqref{lck} holds for some
$1$-form $\theta$,
then $\theta$ is automatically closed, therefore $M$ is l.c.K.
\end{rems}

\medskip

The $1$-form $\theta$ of the previous theorem is called the {\em Lee form} and
it was introduced by H. C. Lee in \cite{L}.
The Lee form is uniquely determined by the following formula:
\begin{equation} \label{tita}
\theta=-\frac{1}{n-1}(\delta\omega)\circ J,
\end{equation}
where $\omega$ is the fundamental $2$-form, $\delta$ is the codifferential and $2n$ is the dimension of $M$.
In general this formula is used to define the Lee
form of any almost Hermitian manifold.
 
\medskip

\begin{ejemplo}
The {\em Hopf manifolds} are examples of locally conformally
K\"ahler manifolds which are not g.c.K. Let $\lambda\in\c$,
$|\lambda|\neq 1$ and $\triangle_\lambda$ be the cyclic group
generated by transformations $z\mapsto\lambda z$ of $\c^n-\{0\}$.
The quotient space $CH_\lambda^n=(\c^n-\{0\})/\triangle_\lambda$ is
a complex manifold and it is called {\em Hopf's complex manifold}.
It can be seen that $CH_\lambda^n$ is diffeomorphic to $S^1\times
S^{2n-1}$. Particularly $CH_\lambda^n$ is compact and its first
Betti number is $b_1(CH_\lambda^n)=1$. Since all odd Betti numbers
of a compact K\"ahler manifold are even, it follows that
$CH_\lambda^n$ cannot admit a K\"ahler metric.

We consider now the Hermitian metric on $\c^n-\{0\}$ \[h=\sum
\frac{dz_j\otimes d\overline z_j}{|z|^2},\] and canonical complex
structure $J$. This metric is invariant by $\triangle_\lambda$, then
it induces a Hermitian metric on $CH_\lambda^n$ which is called the
{\em Boothby's metric}. This Hermitian structure on $CH_\lambda^n$
is in fact l.c.K. and, moreover, $\o$ is parallel with respect to
Levi Civita connection.
The l.c.K. manifolds with this property are a special case of l.c.K. manifolds.
\end{ejemplo}

Let $(M,J,g)$ be an l.c.K. manifold. The metric $g$ on $M$ is called
{\em Vaisman } if the Lee form $\theta$ is parallel with respect to
the Levi-Civita connection of $(M,g)$. A Vaisman manifold is an
l.c.K.~manifold with a Vaisman metric. It is known that Vaisman
manifolds have some special properties which do not necessarily hold
in l.c.K. manifolds. For example, the first Betti number of a
Vaisman manifold is odd (\cite{V3,KS}), while the Oeljeklaus-Toma
manifolds are examples of l.c.K. manifolds with even first Betti
number (\cite{OT}).

\

\subsection{Complex structures on Lie algebras}

\

A left-invariant almost complex structure $J$ on a Lie group $G$ is
a $(1,1)$-differential tensor such that $J_g: T_gG \to T_gG$ is an
endomorphism, $J_g^2=-\I$ for all $g\in G$ and left-traslations are
holomorphic. As usual, the almost complex structure $J$ on $G$ is
called integrable if
\[ [JX,JY]-[X,Y]-J([JX,Y]+[X,JY])=0,\]
for any $X,Y$ vector fields on $G$. In this case, $J$ is called a
left-invariant complex structure and $(G,J)$ is a complex manifold.

A left-invariant (almost) complex structure is determined by its value on the
identity of $G$, and therefore it is possible to
define an (almost) complex structure on the Lie algebra $\g$ of $G$. Namely, a
{\em complex structure J} on a Lie algebra $\g$ is
an endomorphism $J: \g \to \g$ satisfying $J^2=-\I$ and
\[ [JX,JY]-[X,Y]-J([JX,Y]+[X,JY])=0,\] for any $X,Y \in \g$.

\medskip

In this article we will be interested in two special kinds of almost complex
structures on Lie algebras, namely bi-invariant
complex structure and abelian complex structures. An almost complex structure
$J$ on $\g$ is called \textit {bi-invariant} if \[
J[X,Y]=[X,JY], \, \, \text{for all} \, \,  X,Y\in \g,\] and
it is called \textit {abelian} if \[ [JX,JY]=[X,Y], \, \, \text{for all} \, \, 
X,Y\in \g.\]

\medskip
\begin{nota}
Note that in both cases the almost complex structure is automatically
integrable. Also, a complex structure on $\g$ cannot be
bi-invariant and abelian at the same time, unless $\g$ is an abelian Lie
algebra.
 \end{nota}

\medskip

\begin{nota}\label{grupo_complejo}
In general, right-traslations are not holomorphic on a Lie group $G$ with a
left-invariant complex structure $J$. This holds only
when $G$ is a complex Lie group with the holomorphic structure given by $J$, or
equivalently, $J$ is bi-invariant.
\end{nota}

\

Next, we include some properties about abelian complex structures in the
following lemma (see \cite{ABD, BD, P} for their proof).

\begin{lema}\label{prop}
Let $\g$ be a Lie algebra with an abelian complex structure $J$ and $\z(\g)$ its
center. Then
\begin{enumerate}
\item $J\z(\g)=\z(\g)$.
\item $\g' \cap J\g'\subset \z(\g'+J\g')$.
\item The codimension of $[\g,\g]$ is at least $2$, unless $\g$ is isomorphic to
$\aff(\r)$ (the only $2$-di\-men\-sional non-abelian Lie algebra).
\item $\g'$ is abelian, therefore $\g$ is $2$-step solvable.
\end{enumerate}
\end{lema}

\

A rich family of Lie algebras with (abelian) complex structures is
obtained by considering a finite dimensional real associative
algebra $\mathcal A$ and $\aff(\mathcal{A})$ the vector space
$\mathcal{A}\oplus\mathcal{A}$ equipped with the Lie bracket given
by
\[[(a,b),(a',b')]=(aa'-a'a,ab'-a'b), \hspace{10pt} a,b,a',b' \in\mathcal{A}.\]
If $J$ is the endomorphism of $\aff(\mathcal{A})$ defined by \[J(a,b)=(b,-a),
\hspace{10pt} a,b \in \mathcal{A},\]
then it is easy to see that $J$ is a complex structure on $\aff(\mathcal{A})$.
This complex structure is called {\em standard}.
Moreover, when $\mathcal{A}$ is commutative, $J$ is abelian. We prove next a
result about $\aff(\mathcal{A})$ that will be used
later.

\begin{lema}\label{affA}
If $\mathcal{A}$ is an associative commutative algebra and $\aff(\mathcal{A})$
is unimodular, then $\mathcal{A}$ is
nilpotent. Therefore $\aff(\mathcal{A})$ is a nilpotent Lie algebra.
\end{lema}

\begin{proof}
Suppose that $\mathcal{A}$ is not nilpotent, then there exists
$0\neq e\in \mathcal{A}$ such that $e^2=e$. We consider
$(e,0)\in\aff(\mathcal{A})$ and we compute
$\ad_{(e,0)}(x,y)=(0,ey)=(0,l_e(y))$ where $l_e$ is the
left-multiplication by $e$. Therefore the matrix of $\ad_{(e,0)}$ is
of the form
\begin{equation}
\ad_{(e,0)}=\begin{pmatrix} 0 & 0 \cr 0 & l_e\end{pmatrix}
\end{equation}
Since $l_e^2=l_e$ and $l_e\neq 0$, there exists a basis of
$\mathcal{A}$ such that
\begin{equation}
l_e=\begin{pmatrix} I & 0 \cr 0 & 0\end{pmatrix}
\end{equation}
and therefore $\tr(\ad_{(e,0)})\neq0$.
\end{proof}

\smallskip

\begin{nota}
With a similar proof, one can show that if $\mathcal{A}$ is an associative
algebra with
identity, then $\aff(\mathcal{A})$ is not unimodular.
\end{nota}

\

\section{Left invariant l.c.K. metrics on Lie groups}

\

Let $G$ be a Lie group with a left-invariant complex structure $J$ and a
left-invariant metric $g$. If $(G,J,g)$ satisfies the
l.c.K. condition \eqref{lck}, then $(J,g)$ is called a {\em left-invariant l.c.K. structure}
on the Lie group $G$. That is, there exists a
closed $1$-form $\theta$ on $G$ such that $d\omega=\theta\wedge\omega$. We will
see next that the Lee form $\theta$ is left-invariant. Therefore the Lee form $\theta$ is determined by its value in the
identity.

\begin{prop}
Let $G$ be a Lie group with a left-invariant l.c.K. structure $(J,g)$, with
$\theta$ the associated Lee form. Then $\theta$ is
left-invariant.
\end{prop}

\begin{proof}
Recall that if $\alpha$ is any left-invariant form on $G$, then
$d\alpha, \, \ast\alpha$ and $\delta\alpha=\pm\ast\circ \;
d\circ\ast\alpha$ are also left-invariant. Since $J$ is
left-invariant, the claim follows from \eqref{tita}.
\end{proof}

This fact allows us to define l.c.K. structures on Lie algebras.

\medskip

Let $\g$ a Lie algebra, $J$ a complex structure and $\pint$ a Hermitian inner
product on $\g$, with $\w$ its fundamental $2$-form.
The triple $(\g,J,\pint)$ is called {\em locally conformally K\"ahler} (l.c.K.)
if there exists $\o \in \gd$, with $d\o=0$, such
that\footnote{Recall that if $\theta\in\g^*$ and $\omega\in\alt^2\g^*$, then
their exterior derivatives $d\theta\in\alt^2\g^*$
and $d\omega\in\alt^3\g^*$ are given by
\[ d\theta(X,Y)= -\theta([X,Y]), \qquad
d\omega(X,Y,Z)=-\omega([X,Y],Z)-\omega([Y,Z],X)-\omega([Z,X],Y),\]
for any $X,Y,Z\in\g$.}
\begin{equation} \label{g-lck-0}
d\w=\o\wedge\w.
\end{equation}

\smallskip

A Lie algebra $\g$ with a Hermitian structure $(J,\pint)$ is \textit{Vaisman} if
$(\g,J,\pint)$ is l.c.K. and the Lee form is
parallel (see Proposition \ref{ad_A-antisim} below). 

\smallskip

\begin{ejemplo}\label{heisenberg}
Let $\g=\r\times\mathfrak{h}_{2n+1}$, where $\mathfrak{h}_{2n+1}$ is the $(2n+1)$-dimensional
Heisenberg Lie algebra. There is a basis $\{X_1,\dots,X_n,Y_1,\dots,Y_n,Z_1,Z_2\}$ of $\g$ with Lie
brackets given by $[X_i,Y_i]=Z_1$ for $i=1,\dots,n$ and $Z_2$ in the center. We define a metric
$\pint$ on $\g$ such that the basis above is orthonormal. Let $J_0$ be an almost complex
structure given by:
\[J_0X_i=Y_i, \quad  J_0Z_1=-Z_2  \; \; \; \text{for $i=1,\dots,n$}.\]
It is easily seen that $J_0$ is a complex structure on $\g$. Let
$\{x^i,y^i,z^1,z^2\}$ be the $1$-forms dual to $\{X_i,Y_i,Z_1,Z_2\}$
respectively. Then the fundamental form is: 
\[\w=\sum_{i=1}^n(x^i\wedge y^i) - z^1\wedge z^2.\]
Thus $d\w$ is:\[d\w=z^2\wedge\w,\] 
and therefore $(\g,J_0,\pint)$ is l.c.K. It can be seen that the Lee form
$\o=z^2$ is parallel, hence the metric is Vaisman. This example appeared in
\cite{CFL}.

It is known that $\g$ is the Lie algebra of the Lie group $\r\times H_{2n+1}$,
where $H_{2n+1}$ is the group of all matrices with real coefficients which have
the following form:

\[P=\begin{pmatrix} 1& A & c \cr 0& I_n & B^t \cr 0&0&1 \end{pmatrix} \; \; \;
c\in \r, \; I_n=id_{n\times n} .\]
where $A=(a_1,\dots,a_n)\in\r^n$, $B=(b_1,\dots,b_n)\in\r^n$ and $c\in\r$. Let
$\Gamma\subset H_{2n+1}$ be the subgroup of all
matrices with integer coefficients. Then $\Gamma\backslash H_{2n+1}$ is compact
and the nilmanifold $N= S^1 \times
\Gamma\backslash H_{2n+1}$ admits a l.c.K. structure which is Vaisman.
\end{ejemplo}

\begin{rem}
The complex structure $J_0$ defined in the previous example is abelian. Moreover, it was proved in
\cite{BD} that if $J$ is a complex structure on a Lie algebra $\g$ with $\dim\g'=1$, then $J$ is
abelian.
\end{rem}

\medskip

\begin{ejemplo}
In \cite{ACFM} the following example of a l.c.K. solvmanifold was given.
Let $\g$ be the $4$-dimensional solvable Lie algebra given by
\[\g=\text{span}\{A,X,Y,Z\}\]
\[[A,X]=X,\quad [A,Y]=-Y,\quad [X,Y]=Z.\]
Let $\{\alpha,x,y,z\}$ be the dual basis of $\{A,X,Y,Z\}$. We can check by
direct computation that
\[d\alpha=0,\quad dx=-\alpha\wedge x,\quad dy=\alpha\wedge y,\quad dz=-x\wedge
y.\]
Let $\pint$ be a inner product on $\g$ such that $\{A,X,Y,Z\}$ is an orthonormal
basis. If we define $J$ by \[JA=Y, \quad JZ=X,\]
then $(\g,J,\pint)$ is Hermitian with the fundamental $2$-form $\omega$ given by
\[\omega=\alpha\wedge y + z\wedge x.\]
Therefore we obtain
\[d\omega=-\alpha\wedge\omega.\]
Hence it is l.c.K. with Lee form $\theta=-\alpha$. This metric is not Vaisman,
since $\alpha$ is not parallel (see also Lemma
\ref{ad_A-antisim} below). It was proved in \cite{ACFM} that the associated
simply connected solvable Lie group $G$ admits a
lattice $\Gamma$ and therefore the solvmanifold $\Gamma\backslash G$ admits a
l.c.K. structure.
It is proved in \cite{Kam} that this l.c.K. solvmanifold is holomorphically
homothetic to the Inoue surface
$Sol^4_1/\Gamma$ equipped with the locally conformal K\"ahler structure
constructed by Tricerri in \cite{T}.
\end{ejemplo}

\

Now we study some properties about Lie algebras equipped with a l.c.K or Vaisman
structure.

\medskip

Let $(\g,J,\pint)$ be l.c.K. and suppose that $\g$ is not K\"ahler, that is,
$d\omega=\theta\wedge\omega$ where $\theta$ is closed
and $\o\neq 0$. Then the codimension of $\ker\theta$ is $1$ and then we can
choose
\begin{equation}\label{A}
A\in(\ker\theta)^\perp \;\; \text{such that} \;\; \theta(A)=1,
\end{equation}
and therefore $\g$ can be decomposed orthogonally as
\begin{equation} \label{g-lck}
\g=\text{span}\{A\}\oplus \ker\theta, \quad \text{with} \quad
\g'\subset\ker\theta.
\end{equation}
Note that since $\theta\neq 0$, $\g$ cannot be a semisimple Lie algebra.
Since $J$ is skew-symmetric we obtain $\langle JA,A\rangle=0$ and therefore
$JA\in \ker\theta$. If $W$ is the orthogonal complement of $\text{span}\{JA\}$ in $\ker\theta$, we
have
\begin{equation} \label{g-lck-1}
\g=\text{span}\{A,JA\}\oplus^\perp W,
\end{equation}
and $W$ is invariant by $J$.

\medskip

%
%

Note that the Lee form can be written as
\begin{equation}\label{tita-pi}
\theta(X)=\frac{\la X,A\ra}{|A|^2} \; \; \mbox{for all} \; X\in\g.
\end{equation}

\medskip

\begin{lema}\label{Jad_JA-simetrico}
If $(\g,J,\pint)$ is l.c.K. then $J\circ\ad_{JA}$ is symmetric.
\end{lema}

\begin{proof}
For any $X,Y\in \g$ we compute
\begin{align*}
d\omega(JA,X,Y) &=-\omega([JA,X],Y)-\omega([X,Y],JA)-\omega([Y,JA],X)\\
                &=-\langle J[JA,X],Y\rangle-\langle J[X,Y],JA\rangle-\langle
J[Y,JA],X\rangle\\
                &=-\langle J[JA,X],Y\rangle-\langle J[Y,JA],X\rangle.
\end{align*}
On the other hand, using \eqref{tita-pi}, we obtain
\begin{align*}
\theta\wedge\omega(JA,X,Y) & = \theta(X)\omega(Y,JA)+\theta(Y)\omega(JA,X) \\
                &=\frac{\langle A,X\rangle}{|A|^2}\langle
Y,A\rangle-\frac{\langle A,Y\rangle}{|A|^2}\langle A,X\rangle\\
&=0.
\end{align*}
It follows from \eqref{g-lck-0} that $\langle J[JA,X],Y\rangle=\langle
J[JA,Y],X\rangle$ for all $X,Y \in \g$, hence $J\circ
\ad_{JA}$ is symmetric.
\end{proof}

\

Now we consider a Vaisman structure $(J,\pint)$ on $ \g$, with $\o$ its parallel
Lee form. In this context the Vaisman
condition can be characterized as follows.

\begin{lema}\label{ad_A-antisim}
Let $(\g,J,\pint)$ be l.c.K. and let $A\in\g$ be as in \eqref{A}. Then $(\g,
J,\pint)$ is Vaisman if and only if $\ad_A$ is skew-symmetric.
\end{lema}

\begin{proof}
Recall first that the Levi-Civita connection $\nabla$ of a left-invariant
Riemannian metric on a Lie group is itself
left-invariant, that is, $\nabla_XY$ is a left-invariant vector field whenever
$X,Y$ are left-invariant. Similarly,
$\nabla T$ is a left-invariant tensor if $T$ is a left-invariant tensor. In
particular, if $\eta$ is a left-invariant $1$-form,
we have that $(\nabla_X\eta)(Y)=-\eta(\nabla_XY)$ for $X,Y$ left-invariant
vector fields.

Let us now compute $\nabla \theta$, where $\nabla$ is the Levi-Civita connection
on $\g$ associated to $\pint$ and $\theta$ is
the Lee form. Given $X,Y\in \g$ we have that
\[ (\nabla_X \theta)(Y)=
-\o(\nabla_XY)=-\frac{\langle\nabla_XY,A\rangle}{|A|^2}=-\frac{1}{2|A|^2}\{
\langle[X,Y],A\rangle-\langle[Y,A],X\rangle+\langle[A,X
] ,Y\rangle\}\]
Since $A$ is orthogonal to $\g'$ we obtain:
\[(\nabla_X
\theta)(Y)=-\frac{1}{2|A|^2}\{\langle[A,Y],X\rangle+\langle[A,X],Y\rangle\}\]
Therefore $(\nabla_X \theta)(Y)=0$ if and only if
$\langle[A,Y],X\rangle=-\langle[A,X],Y\rangle$ for all $X,Y\in \g$.
\end{proof}

\medskip

\begin{nota}
Let $H$ be a Lie group equipped with a left-invariant metric $h$. Recall that
the endomorphism $\ad_A$ of its Lie algebra
$\mathfrak h$ is skew-symmetric with respect to $h_e$ if and only if the
left-invariant vector field on $H$ determined by $A$ is
Killing.
\end{nota}

\

\section{L.c.K. structures with bi-invariant complex structure}

\

In this section we consider a Lie algebra equipped with an l.c.K.
structure such that its complex structure is bi-invariant.
Equivalently, we are considering left-invariant l.c.K. metrics on
complex Lie groups. The aim of this section is to prove the
following result, where we show that in each even (real) dimension,
there is only one Lie algebra admitting such metrics.

\

\begin{teo}\label{J-bi}
Let $(J,\pint)$ be an l.c.K. structure on $\g$ with a bi-invariant complex
structure $J$. Then $\g \simeq
\r^2\ltimes\r^{2n}$, a $J$-invariant and orthogonal sum, where $\r^2$ is
generated by $A,B$ with $[A,B]=0$, $JA=B$,
$[A,X]=-\frac12 X$ and $[B,X]=-\frac12 JX$ for all $X\in\r^{2n}$.
\end{teo}

\medskip

In order to prove this theorem, we recall the following well known result
concerning the existence of K\"ahler metrics on complex Lie groups.
\begin{lema}\label{bi-K}
If $(\g,J,\pint)$ is K\"ahler with $J$ bi-invariant, then $\g$ is abelian.
\end{lema}

\

\begin{proof}[Proof of Theorem]
The fact that $J$ is bi-invariant implies that \[J\g'=\g'.\]
Recall from \eqref{g-lck} the orthogonal decomposition
\[\g=\text{span}\{A\}\oplus \ker\theta.\]
\begin{lema}
The endomorphism $\ad_A: \g \to \g$ is symmetric.
\end{lema}
\begin{proof}
Since $J$ is bi-invariant we have that $J\circ\ad_{JA}=-\ad_A$, and
it follows from Lemma \ref{Jad_JA-simetrico} that $\ad_A$ is
symmetric.
\end{proof}

Next, from \eqref{g-lck-1} we have
\[\g=\text{span}\{A,JA\}\oplus^\perp W,\]
where $\g'\subset\ker\theta=\text{span}\{JA\}\oplus W$ and $W$ is $J$-invariant.
It is easy to see that $\g'\subset W$. Actually
$\g'=W$, since for $X\in W$ we have
$d\omega(A,X,JX) = -2\langle [A,X],X\rangle$ and $\theta\wedge\omega(A,X,JX) =
|X|^2$, therefore
\begin{equation}
-2\langle [A,X],X\rangle=|X|^2,
\end{equation}
this implies that $\g'=W$. Then we obtain
\[\g=\text{span}\{A,JA\}\oplus^\perp \g',\]
with $\g'$ $J$-invariant.

On the other hand $(\g',J|_{\g'},\pint)$ is K\"ahler, since the fundamental form
of $\g'$ is the restriction of $\omega$ to
$\g'\times\g'$, and $d\omega=0$ on $\g'$. From Lemma \ref{bi-K} we get that
$\g'$ is abelian.

Since $\ad_A$ is symmetric and $d\omega(A,X,JY) = \theta\wedge\omega(A,X,JY)$,
we have that
\[2\langle [A,X],Y\rangle=-\langle X,Y\rangle, \quad \text{for } X,Y\in \g'.\]
Therefore $[A,X]=-\frac12 X$ for all $X\in\g'$. Setting $B=JA$, we
obtain $[A,B]=0$, thus $\g =\r^2\ltimes\r^{2n}$ where
$\ad_A|_{\r^{2n}}=-\frac12 \I$ and $\ad_B|_{\r^{2n}}=J\ad_A=-\frac12
J$.
\end{proof}

\medskip

\begin{cor}
There exists no unimodular Lie algebra $\g$ with an l.c.K. structure $(J,\pint)$ such that $J$ is a
bi-invariant complex
structure.
\end{cor}

Let $M$ be a compact complex parallelizable manifold. According to \cite{W}, $M$ may be written as a
quotient $\Gamma\backslash G$, where $G$ is a simply connected complex Lie group and $\Gamma$ is a
discrete subgroup. Note that according to \cite{Mi}, $G$ must be unimodular. Let $\pi:G\to M$ denote
the holomorphic projection.

\begin{cor}
With notation as above, $M$ does not admit any l.c.K. metric $g$ compatible with
its holomorphic structure such that $\pi^*g$ is
a left-invariant metric on $G$.
\end{cor}

\begin{nota}
In \cite{HK}, it is proved more generally that a compact complex parallelizable
manifold does not admit any l.c.K. metric
compatible with its holomorphic structure.
\end{nota}

\

\section{L.c.K. structures with abelian complex structure}

\

In this section we consider a Lie algebra equipped with an l.c.K. structure such that its complex
structure is abelian. Our aim is to prove the following result, where we show that the only
unimodular Lie algebras that admit such metrics are the product of a Heisenberg Lie algebra by $\r$,
and the l.c.K. structure is in fact Vaisman. From now on we assume that the Lie algebras we work
with are at least $4$-dimensional.

\

Before stating the main result, we consider the following variation of Example \ref{heisenberg}.
Recall that $\r\times\mathfrak{h}_{2n+1}$ has a basis $\{X_1,\dots,X_n,Y_1,\dots,Y_n,Z_1,Z_2\}$
such that $[X_i,Y_i]=Z_1$ for $i=1,\dots,n$, and that this Lie algebra admits an abelian complex
structure $J_0$ given by $J_0X_i=Y_i,\, J_0Z_1=-Z_2$. For any $\lambda>0$, consider the metric
$\pint_\l$ such that the basis above is orthogonal, with $|X_i|=|Y_i|=1$ but
$|Z_1|^2=|Z_2|^2=\frac{1}{\lambda}$. It is easy
to see (just as in Example \ref{heisenberg}) that $(J_0,\pint_\l)$ is an l.c.K. structure, in fact,
it is Vaisman. Furthermore, the metrics $\pint_\l$ are pairwise non-isometric, since the scalar
curvature of $\pint_\l$ is $-\frac{n\lambda^2}{2}$. 

\

\begin{teo}\label{main-theorem}
Let $(J,\pint)$ be an l.c.K. structure on $\g$ with abelian complex structure $J$. If $\g$ is
unimodular then $\g \simeq \r\times\h_{2n+1}$, where $\mathfrak{h}_{2n+1}$ is the
$(2n+1)$-dimensional Heisenberg Lie algebra, and $(J,\pint)$ is equivalent to $(J_0,\pint_\l)$ for
some $\l>0$.
\end{teo}

\

We will provide the proof of this theorem in a series of results. Recall from
\eqref{g-lck} that
\[ \g=\text{span}\{A\}\oplus \ker\theta, \]
where $A\in(\ker\theta)^\perp$ such that $\theta(A)=1$.

\begin{lema}\label{ad_A-sim}
The endomorphism $\ad_A: \g \to \g$ is symmetric.
\end{lema}
\begin{proof}
It is an immediate consequence of Lemma \ref{Jad_JA-simetrico}.
\end{proof}

\

In particular $\ad_A|_{\ker\o}: \ker\o\to\ker\o$ is symmetric, therefore it is
diagonalizable and thus we get the following
decomposition:
\[\ker\o=\sum_{\l\in S} \gl,\]
where $S$ is the spectrum of $\ad_A|_{\ker\o}$ and $\gl$ is the eigenspace
associated with the eigenvalue $\l$.

According to Lemma \ref{prop} $\riii$ the codimension of $[\g,\g]$ is at least
2. Therefore $\go\neq\{0\}$, that is $0\in S$. Hence we obtain the following
orthogonal decomposition:
\begin{equation}\label{dec-g}
\g=\r A \oplus \go \oplus \sum_{\l\in S^*} \gl,
\end{equation}
where $S^*:=S-\{0\}.$ Note that the Jacobi's identity together with
the fact that $\g'$ is abelian imply that:
\begin{itemize}
 \item $\gl$ is an ideal for $\l\in S^*$.
 \item $\go$ is a subalgebra.
\end{itemize}


Now we consider $\gcc=[\go, \go]$ and $\gccc$, its orthogonal
complement in $\go$, that is, \[\go=\gcc \oplus \gccc.\] Note also
that $\displaystyle{\g'=\gcc\oplus\sum_{\l\in S^*}\gl}.$

For any $X,Y \in \ker\o$, we compute
\begin{align*}
d\omega(A,X,Y) &= -\omega([A,X],Y)-\omega([X,Y],A)-\omega([Y,A],X)\\
&= -\langle J[A,X],Y\rangle -\langle J[X,Y],A\rangle-\langle J[Y,A],X\rangle \\
&= \langle [A,X],JY\rangle+\langle [X,Y],JA\rangle+\langle [Y,A],JX\rangle
\end{align*}
\begin{align*}
\theta\wedge\omega(A,X,Y) &=
\theta(A)\omega(X,Y)+\theta(X)\omega(Y,A)+\theta(Y)\omega(A,X) \\
&=\langle JX,Y\rangle.
\end{align*}
Therefore we have
\[\langle [A,X],JY\rangle+\langle [X,Y],JA\rangle+\langle [Y,A],JX\rangle
=\langle JX,Y\rangle,\] for all $X,Y \in \ker\o$.
From this we get three equations that will be important later:
\begin{equation}\label{eq1}
\langle [X,Y],JA\rangle=\langle JX,Y\rangle, \; \mbox{for $X,Y \in \go$}.
\end{equation}
\begin{equation}\label{eq2}
(\l + \u + 1)\langle JX,Y\rangle=0, \; \mbox{for $X\in\gl, Y\in\gm$
and $\l,\u\in S^*$}.
\end{equation}
\begin{equation}\label{eq3}
\langle [X,Y],JA\rangle=(\u + 1)\langle JX,Y\rangle, \; \mbox{for
$X\in\go, Y\in\gm$,  $\u\in S^*$}.
\end{equation}

\medskip

\begin{lema}\label{Jg0c}
$J(\gcc)\subset \r A\oplus\gccc.$
\end{lema}

\begin{proof}
If $X,Y \in \go$ we can write $\displaystyle{J[X,Y]=aA+Z_0+\sum_{\l\in
S^*}Z_\l}$ with $a\in\r$, $Z_0\in\go$ and $Z_\l\in\gl$. Then $\displaystyle{[A,
J[X,Y]]=[A,\sum_{\l\in S^*}Z_\l]=\sum_{\l\in S^*}\l Z_\l}$.

On the other hand, from \eqref{dec-g},
$JA=X_0 + \displaystyle{\sum_{\l\in S^*}X_\l},$
with $X_0\in\go$ y $X_\l \in \gl$ for all $\l\in S^*$. Then \[[A, J[X,Y]]=-[JA,
[X,Y]]=-[X_0 + \displaystyle{\sum_{\l\in S^*}X_\l}, [X,Y]]=-[X_0,[X,Y]] \in
\go\] since $\go$ is subalgebra and $\g'$ is abelian. Therefore $Z_\l=0$  for
all $\l\in
S^*$.

Moreover, it follows from \eqref{eq1} and the fact that $\g'$ is abelian that
$J(\gcc)$ and $\gcc$ are orthogonal. Then we must have $Z_0\in\gccc$, and this
implies the result.
\end{proof}

Now, we define $\L\subset S^*$ in the following way: $\l\in\L$ if
and only if there is not any $\l'\in S^*$ such that $\l+\l'+1=0$, or
equivalently, $\L=\{\l\in S^*: -(\l+1)\notin S^*\}$. Note that
$\l\notin\L$ if and only if $-(\l+1)\notin\L$. 

\medskip

\begin{lema}\label{Jgl}
Let $\l\in S^*$. Then,
\begin{enumerate}
\item [\ri] if $\l\in\L$ then $J(\gl)\subset \r A\oplus\gccc.$
\item [\rii] if $\l\in\L^c$ then $J(\gl)\subset \r A\oplus\gccc\oplus\gll,$
where $\l+\ll+1=0$.
\end{enumerate}
\end{lema}

\begin{proof}
$\ri$ If $\l\in\L$, from \eqref{eq2} we have that $J(\gl)$ is orthogonal to
$\g_\mu$ for all $\mu\in S^*$, and therefore $J(\gl)\subset \r A\oplus\g_0.$
Moreover, for $X_\lambda\in\gl, Y\in\g_0'$, it follows from Lemma \ref{Jg0c}
that $\langle JX_\lambda, Y\rangle=-\langle X_\lambda, JY\rangle=0$. This proves
$\ri$, and in a similar way we prove $\rii$. 
\end{proof}

\medskip

Let $\h$ be the orthogonal complement of $\g' + J\g'$ in $\g$. Note that $\h$ is
$J$-invariant. Thus, we can write
\begin{equation}\label{g'-Jg'-h}
\g=(\g' + J\g')\oplus\h.
\end{equation}

We will show that this orthogonal complement $\h$ is non-zero. We begin with an auxiliary result.

\begin{lema}\label{int}
$\displaystyle{\g'\cap J\g'=\sum_{\L^c}\gl \cap J\left(\sum_{\L^c}\gl\right)}$.
\end{lema}
\begin{proof}
Given $Y\in\g' \cap J\g'$ then $Y=JZ$ for $Z\in\g'$. Since
$\displaystyle{\g'=\gcc\oplus\sum_{\l\in S^*}\gl}$, we can write
\[Y=Y_0+\sum_{\l\in S^*}Y_\l, \;\; Z=Z_0+\sum_{\l\in S^*}Z_\l\] with $Y_0,
Z_0\in\gcc$ and $Y_\l, Z_\l\in\gl$. Then $\displaystyle{JZ=JZ_0+\sum_{\l\in
S^*}JZ_\l}$.
From Lemma \ref{Jg0c} and Lemma \ref{Jgl} we obtain
$\displaystyle{JZ\in\r A\oplus\gccc\oplus\sum_{\l\in\L^c}\gl}$.
Since $\displaystyle{JZ=Y\in\gcc\oplus\sum_{\l\in S^*}\gl}$, we get
$Y_0=0$ and $Y_\l=0$ for all $\l\in\L$, therefore
$\displaystyle{Y\in\sum_{\l\in\L^c}\gl}$. In the same way,
$\displaystyle{Z\in\sum_{\l\in\L^c}\gl}$, and therefore
$\displaystyle{\g'\cap J\g' \subset \sum_{\L^c}\gl \cap
J\left(\sum_{\L^c}\gl\right)}$. The other inclusion is clear.
\end{proof}

\medskip

\begin{lema}
 With notation as above, $\h\neq 0$.
\end{lema}

\begin{proof}
If we suppose that $\h=\{0\}$ we get from \eqref{g'-Jg'-h} that 
\[ \g=\g' + J\g'. \]
\underline{Claim}: $\g' \cap J\g'=\{0\}$. \\
Indeed, according to Lemma \ref{int} and to Lemma \ref{prop} $\rii$ we have $\g' \cap
J\g'= \displaystyle{\sum_{\L^c}\gl} \cap
J\left(\sum_{\L^c}\gl \right)\subset \z(\g).$
Given $Y\in\g' \cap J\g'$, it can be written as $\displaystyle{Y=\sum_{\l\in\L^c} Y_\l}$. Then
$0=[A,Y]=\displaystyle{\sum_{\l\in\L^c}\l Y_\l}$, and therefore $Y=0$. This proves the claim.

\medskip

As a consequence, we have the orthogonal decomposition \[\g=\g'\oplus
J\g'.\] 
According to \cite[Corollary 3.3]{ABD1}, the Lie bracket on $\g$ induces a
structure of commutative associative algebra on $\g'$ given by $X \ast Y = [ J
X, Y ]$. Furthermore if $\mathcal{A}$ denotes the
commutative associative algebra $(\g',\ast)$, then $\mathcal{A}^2=\mathcal{A}$
and $\g$ is holomorphically isomorphic to
$\aff(\mathcal{A})$ with its standard complex structure (see Section 2.2). Since $\g$ is unimodular,
it follows from Lemma \ref{affA} that
$\mathcal{A}$ is nilpotent. This is a contradiction with the fact
that $\mathcal{A}^2=\mathcal{A}$ and therefore it must be
$\h\neq\{0\}$.

\end{proof}

\pagebreak

Since $A$ is orthogonal to $\g'$, we have that $JA$ is orthogonal to $J\g'$.
More precisely, we have

\begin{lema}\label{JA}
$JA\in\g'$.
\end{lema}
\begin{proof}
Let $\mathfrak{u}$ be the orthogonal complement of $\g'\cap J\g'$ in $\g'$, that
is \[\g'=\mathfrak{u}\oplus(\g' \cap J\g').\]
Since $\g'\cap J\g'$ is $J$-invariant we have \[ J\g'=J\mathfrak{u}\oplus
(\g'\cap J\g'),\] and therefore \[\g'+J\g'=\mathfrak{u}\oplus(\g'\cap
J\g')\oplus J\mathfrak{u}.\]
As $JA$ is orthogonal to $J\g'$, it follows from \eqref{g'-Jg'-h} that $JA= U
+\b$ for some $U\in\mathfrak{u}$ and $\b\in\h$. Since $\mathfrak{u}\subset\g'$, the lemma will
follow if we prove $\b=0$.

Now, for any $X\in\g$ such that $\la A,X\ra=\la JA,X\ra=0$, we compute
\begin{align}
d\w(J\b,X,JX) &= \o\wedge\w(J\b,X,JX)\nonumber\\
-\la[J\b,X],X\ra +\la[X,JX],-\b]\ra +\la[JX,J\b],JX\ra&= \frac{\la
A,J\b\ra}{|A|^2}|X|^2\nonumber\\
\la[J\b,X],X\ra +\la[J\b,JX],JX\ra &= \frac{|\b|^2}{|A|^2}|X|^2. \label{tr}
\end{align}
Moreover $\la[J\b,A],A\ra=0$ and $\la[J\b,JA],JA\ra=0$, due to Lemma
\ref{Jad_JA-simetrico} and the fact that $\h$ is orthogonal to $\g'$. Let
$\{X_1,JX_1,\dots,X_r,JX_r\}$ be an orthonormal basis of $W$, where $W$ is given in
\eqref{g-lck-1}. Note that $r\geq 1$ since $\dim \g\geq 4$. We compute next $\tr(\ad_{J\b})$, taking
into account \eqref{tr}:
\begin{align*}
\tr(\ad_{J\b})&=\frac{1}{|A|^2}\la[J\b,A],A\ra+\frac{1}{|A|^2}\la[J\b,JA],JA\ra+\sum_{j=1}^r\la[J\b,
X_j],
X_j\ra +\la[J\b,JX_j],JX_j\ra  \\
&=\frac{|\b|^2}{|A|^2} \sum_{j=1}^r |X|^2.
\end{align*}
Since $\g$ is unimodular, it follows that $\b=0$.
\end{proof}

\medskip

\begin{nota}
It follows from Lemma \ref{JA} and \eqref{g'-Jg'-h} that if $H\in\h$, then $H$ is orthogonal
to $A$ and $JA$. 
\end{nota}

\

\begin{lema}\label{feo}
If $H\in\h$, then 
\begin{enumerate}
 \item[\ri] $ \langle [H,JH],JA\rangle = |H|^2$,
 \item[\rii] $|[H,JH]|^2=\frac{|H|^4}{|A|^2}$.
\end{enumerate}
\end{lema}

\begin{proof}
For $H\in\h$, we compute first
\begin{align*}
d\w(A,H,JH) &= \o\wedge\w(A,H,JH) \nonumber\\
\langle [A,H],J^2H\rangle+\langle [H,JH],JA\rangle+\langle [JH,A],JH\rangle &=
|H|^2-\frac{\la A,H\ra^2}{|A|^2}-\frac{\la
JA,H\ra^2}{|A|^2} \nonumber\\
\langle [H,JH],JA\rangle &= |H|^2, 
\end{align*}
since $\h$ is $J$-invariant and orthogonal to $\g'$. This proves (i).

Now we compute
\begin{align*}
d\w(J[H,JH],H,JH) &= \o\wedge\w(J[H,JH],H,JH) \nonumber\\
-|[H,JH]|^2 &= \frac{\la A,J[H,JH]\ra}{|A|^2}|H|^2 \nonumber\\
 |[H,JH]|^2 &= \frac{|H|^4}{|A|^2}, 
\end{align*}
where we used (i) for the last equality. This proves (ii).
\end{proof}

\medskip

\begin{lema}\label{feo1}
If $H\in\h$, then
\begin{enumerate}
 \item [\ri] $[H,JH]=\frac{|H|^2}{|A|^2}JA$,
 \item [\rii] $[H,\gcc]=0$,
 \item [\riii] $[H,\gl]=0$ for all $\l\in S^*-\{-\frac{1}{2}\}$,
 \item [\riv] $[H,Y]=0$ for all $Y\in\h$ such that $\langle Y,JH\rangle=0$.
 \end{enumerate}
\end{lema}

\begin{proof}
(i) Using Lemma \ref{feo} and the Cauchy-Schwarz inequality we obtain
\[|H|^4=\langle [H,JH],JA\rangle^2\leq
|[H,JH]|^2|A|^2=\frac{|H|^4}{|A|^2}|A|^2=|H|^4,\]
so that we have equality everywhere and therefore for all $H\in\h$ there exists
$c(H)\in\r$ such that \[[H,JH]=c(H)JA.\]
From Lemma \ref{feo} $\rii$ again we get that $|H|^2=c(H)|A|^2$, and therefore
$[H,JH]=\frac{|H|^2}{|A|^2}JA$ for all $H\in\h$.

\

Both (ii) and (iii) will follow from the next computation. Given $H\in\h, X,Y\in \g'$ we compute
\begin{align*}
d\w(H,X,JY) &= -\w([H,X],JY)-\w([X,JY],H)-\w([JY,H],X)\\
&= -\la J[H,X],JY \ra -\la J[X,JY],H \ra -\la J[JY,H],X \ra\\
&= -\la [H,X],Y \ra +\la [X,JY],JH \ra +\la [JY,H],JX \ra\\
&= -\la [H,X],Y \ra +\la [JY,H],JX \ra,
\end{align*}
since $\h$ is $J$-invariant and orthogonal to $\g'$. On the other hand we have
\begin{align*}
\o\^\w(H,X,JY) &= \o(H)\w(X,JY)+\o(X)\w(JY,H)+\o(JY)\w(H,X)\\
&= \frac{\la A,H\ra}{|A|^2}\la X,Y\ra -\frac{\la A,X\ra}{|A|^2}\la Y,H\ra
+\frac{\la A,JY\ra}{|A|^2}\la JH,X\ra\\
&= 0,
\end{align*}
since $\la H,A \ra=0$ and $\la\h,\g'\ra=0$. Therefore we get $\la [H,X],Y \ra
=\la [JY,H],JX \ra$. In particular, if we take
$Y=[H,X]$ we obtain 
\begin{equation}\label{nose}
 |[H,X]|^2 = \la [J[H,X],H],JX \ra.
\end{equation}

$\rii$ If $X\in\go'$ it follows from Lemma \ref{Jg0c} that $JX\in \r
A\oplus\gccc$. Since $[J[H,X],H]\in\g'$ we get from \eqref{nose} that $|[H,X]|^2=0$.

$\riii$ If $X_\l\in\gl, \l\in\L,$ it follows from Lemma \ref{Jgl} \ri \, that
$JX\in \r A\oplus\gccc$. In the same way as above
we get $|[H,X]|^2=0$.

However, if $X_\l\in\gl, \l\in\L^c$ and $\l\neq -\frac{1}{2}$ from Lemma
\ref{Jgl} \rii \, we obtain that $JX\in \r A\oplus\gccc\oplus\gll$, where
$\ll=-\l-1$ and $\ll\neq\l$ since $\l\neq-\frac{1}{2}$. On the other hand
$[J[H,X],H]=-[[H,X],JH]\in\gl$ since $\gl$ is an ideal. Therefore from \eqref{eq2} and \eqref{nose}
we get $|[H,X]|^2=0$.

\

(iv) Finally, we calculate $[H,Y]$ for $Y\in\h$ such that $\la Y,JH\ra=0$.
\begin{align*}
d\w(J[H,Y],H,Y) & = \la[J[H,Y],H],JY\ra - |[H,Y]|^2 + \la[Y,J[H,Y]],JH\ra \\
                & = -|[H,Y]|^2, 
\end{align*}
since $\h$ is $J$-invariant and orthogonal to $\g'$. On the other hand
\[\o\wedge\w(J[H,Y],H,Y)= \frac{\la A,J[H,Y]\ra}{|A|^2}\la JH,Y\ra + \frac{\la
A,H\ra}{|A|^2}\la JY, J[H,Y]\ra -\frac{\la
A,Y\ra}{|A|^2}\la [H,Y], H\ra=0,\]
since $A$ is orthogonal to $\h$ and $\la JH,Y\ra=0$. Therefore $[H,Y]=0$.

\end{proof}

\medskip

\begin{prop}\label{A-central}
With notation as above, we have:
\begin{itemize}
 \item [\ri] $A, JA \in \z(\g)$,
 \item [\rii] $\g=\g'\oplus J\g'\oplus\h,$ an orthogonal sum.
\end{itemize}
\end{prop}

\begin{proof}

$\ri$ Let $\l\in S^*-\{-\frac12\}$ and take $H\in\h, H\neq0$ and $X_\l\in\gl$.
Lemma \ref{feo1} (i) implies that
\[ [[H,JH],JX_\l]=\frac{|H|^2}{|A|^2}[JA,JX_\l]=\frac{|H|^2}{|A|^2}\l X_\l,\]
whereas Lemma \ref{feo1} $\rii$ and the fact that $\h$ is $J$-invariant imply that 
\[[[H,JH],JX_\l]=-[[JH,JX_\l],H]-[[JX_\l,H],JH]=0.\]

Then $X_\l=0$ and therefore $S^*-\{-\frac12\}=\emptyset$. If
$-\frac{1}{2}\in S^*$ then it is the only eigenvalue in $S^*$,
hence $\g$ is not unimodular, that is a contradiction. As a consequence,
$S^*=\emptyset$, that is, $S=\{0\}$, or equivalently, $A\in \z(\g)$. It follows from Lemma
\ref{prop} that $JA\in\z(\g)$ too.

\

$\rii$ It follows from Lemma \ref{int} that $\g'\cap J\g'=\{0\}$.
Therefore \[\g=\g'\oplus J\g'\oplus\h\] where $\g'=\go'$. Moreover,
this decomposition is orthogonal, because of Lemma \ref{Jg0c}.
\end{proof}

\medskip 

\begin{nota}
If $(\g,J,\pint)$ is Vaisman with $J$ abelian, it is much easier to
show that $A, JA\in\z(\g)$. Indeed, from Lemma \ref{ad_A-sim} and
Lemma \ref{ad_A-antisim} we have that $\ad_A$ is symmetric and
skew-symmetric, so $A\in\z(\g)$. Then $J$ abelian implies that
$JA\in\z(\g)$ too.
\end{nota}

\
\begin{proof}[Proof of Theorem \ref{main-theorem}]

Recall that we have the following orthogonal decompositions of $\g$,
\[\g=\text{span}\{A\}\oplus\go=\g'\oplus J\g'\oplus \h, \] 
with $JA\in \g'$, where $\g'$ and $J\g'$ are abelian subalgebras.
Since $\g_0'=\g'$, it follows from Lemma \ref{feo1} that $[\h,\g']=[\h,J\g']=0$ and, moreover, the
Lie bracket on $\h$ is also determined. In order to characterize completely the Lie bracket on
$\g$, we only need to consider $[\g',J\g']$, that is, the brackets $[X,JY]$ for $X,Y\in\g'$. Since
$A,JA\in\z(\g)$ we may assume that $X,Y$ are orthogonal to $JA$.

\begin{prop}
Using the previous notation, we have that:
\begin{itemize}
 \item [\ri] $[X,JY]=0$ for $\la X,Y\ra=0$,
 \item [\rii]$[X,JX]=\frac{|X|^2}{|A|^2}JA$.
\end{itemize}
\end{prop}

\begin{proof}
We calculate
\[ d\w(J[X,JY],X,JY)=-\la[J[X,JY],X],Y\ra -\la[X,JY],[X,JY]\ra +
\la[JY,J[X,JY]],JX\ra.\]
Since $\g'$ and $J\g'$ are orthogonal, we get $\la[JY,J[X,JY]],JX\ra=0$. From Jacobi identity and
the fact that $\g'$ is abelian we have that
\[[J[X,JY],X]=-[[X,JY],JX]=[[JY,JX],X]+[[JX,X],JY]=[[JX,X],JY].\]
Therefore $d\w(J[X,JY],X,JY)= -|[X,JY]|^2+\la\ad_{J[JX,X]}Y,Y\ra$. On the other hand,
\[\o\wedge\w(J[X,JY],X,JY)=\frac{\la A,J[X,JY]\ra}{|A|^2}\la X,Y\ra - \frac{\la
A,X\ra}{|A|^2}\la Y,J[X,JY]\ra - \frac{\la
A,JY\ra}{|A|^2}\la[X,JY],X\ra.\]
Since $\g'$ and $J\g'$ are orthogonal and the fact that $\la Y,JA\ra=0$ we get
$\o\wedge\w(J[X,JY],X,JY)=\frac{\la
A,J[X,JY]\ra}{|A|^2}\la X,Y\ra$. It follows from \eqref{eq1} that
$\o\wedge\w(J[X,JY],X,JY)=-\frac{\la X,Y\ra^2}{|A|^2}$.
Therefore
\[|[X,JY]|^2-\la\ad_{J[JX,X]}Y,Y\ra=\frac{\la X,Y\ra^2}{|A|^2}.\]

If $\la Y,X\ra=0$, we have
\begin{equation}\label{N1}
\la\ad_{J[JX,X]}Y,Y\ra=|[X,JY]|^2.
\end{equation}

If $Y=X$, we get
\begin{equation}\label{N2}
\la\ad_{J[JX,X]}X,X\ra=|[X,JX]|^2-\frac{|X|^4}{|A|^2}.
\end{equation}

Our aim is to calculate $\tr(\ad_{J[JX,X]})$ for a fixed $X\in\g'$ and $\la
X,JA\ra=0$.

It follows from \eqref{eq1} that $\la[X,JX],JA\ra=\la JX,JX\ra=|X|^2.$ Then
using the Cauchy-Schwarz inequality we obtain
\begin{equation}\label{N3}
|X|^4=|\la[X,JX],JA\ra|^2\leq |[X,JX]|^2|A|^2=\la\ad_{J[JX,X]}X,X\ra|A|^2+|X|^4.
\end{equation}

From \eqref{N1} and \eqref{N3} we have that $\la\ad_{J[JX,X]}X,X\ra\geq0$ and
$\la\ad_{J[JX,X]}Y,Y\ra\geq0$ for all $Y\in\g'$,
$\la Y,X\ra=0$ and $\la Y,JA\ra=0$.

We know also that $\ad_{J[JX,X]}JA=0$, $\ad_{J[JX,X]}JZ=0$ for all $Z\in\g'$ and
$\ad_{J[JX,X]}H=0$ for all $H\in\h$.
Therefore
\begin{equation}\label{tra}
\tr(\ad_{J[JX,X]})=\la\ad_{J[JX,X]}X,X\ra+\sum_{\la
Y,X\ra=0}\la\ad_{J[JX,X]}Y,Y\ra.
\end{equation}

Since $\g$ is unimodular, it must be $\la\ad_{J[JX,X]}Y,Y\ra=0$, and from
\eqref{N1} we get $[X,JY]=0$ for all $Y\in\g', \la
Y,X\ra=0$. This proves $\ri$.

\

$\rii$ From \eqref{tra} we have that $\la\ad_{J[JX,X]}X,X\ra=0$. Then it follows
from \eqref{N3} that $[X,JX]=c(X)JA$ with
$c(X)\in\r$. From \eqref{eq1} we get $\la[X,JX],JA\ra=|X|^2$, and this implies
that $c(X)=\frac{|X|^2}{|A|^2}.$ Therefore
$[X,JX]=\frac{|X|^2}{|A|^2}JA$ for all $X\in \g'$.
\end{proof}

\smallskip

As a consequence, the only non-vanishing brackets on $\g$ are $[X,JX]=\frac{|X|^2}{|A|^2}JA$ for
$X\in\g$ with $\la X,A\ra=0$ and $\la X,JA\ra=0$, so that $\g'=\text{span}\{JA\}$ and
$J\g'=\text{span}\{A\}$. Considering an orthonormal basis $\{X_1,\ldots,X_n,Y_1,\ldots,Y_n\}$ of
$\h$, with $JX_i=Y_i$, we have that the only non-vanishing brackets on $\g$ are
$[X_i,Y_i]=\frac{JA}{|A|^2}$. Setting $Z_1=\frac{JA}{|A|^2},\, Z_2=\frac{A}{|A|^2}$, it is clear that
$\g$ is isomorphic to $\r\times\mathfrak{h}_{2n+1}$ and $(J,\pint)$ is equivalent to
$(J_0,\pint_\l)$ for $\lambda=|A|$.
\end{proof}

\

\begin{nota}
On $\r\times {\mathfrak h}_{2n+1}$ there are $\left[\frac{n}{2}\right]+1$
equivalence classes of complex structures, all of them
abelian (see \cite[Proposition 2.2]{ABD}). It follows from the proof of Theorem
\ref{main-theorem} that if $(J,\pint)$ is an l.c.K. structure on this Lie algebra, then $J$
is equivalent to the complex structure $J_0$, so that representatives of only one equivalence class
of complex structures may admit l.c.K. metrics (compare \cite{U}).
\end{nota}

\medskip

In terms of solvmanifolds, we can rewrite Theorem \ref{main-theorem} as follows.

\begin{cor}
Let $\Gamma\backslash G$ be a compact solvmanifold with an l.c.K. structure induced from a left
invariant l.c.K. structure on $G$ with an abelian complex structure, and $G$ simply connected.
Then $G$ is isomorphic to $\r \times H_{2n+1}$, and $\r \times H_{2n+1}$ has a left-invariant l.c.K.
structure induced by $(J_0, \pint_\l)$ for some $\lambda>0$. In particular, $\Gamma\backslash G$ is
a nilmanifold and the l.c.K. structure is Vaisman.
\end{cor}

\


\begin{thebibliography}{99}\frenchspacing

\bibitem{ABD} {\sc A. Andrada, M. L. Barberis, I. G. Dotti}, Abelian Hermitian geometry,
\textit{Diff. Geom. Appl.} {\bf 30} (2012) 509--519.

\bibitem{ABD1} {\sc A. Andrada, M. L. Barberis, I. G. Dotti}, Classification of abelian complex
structures on 6-dimensional Lie algebras, \textit{J. London Math. Soc.} {\bf 83} (2011), 232--255.

\bibitem{ACFM} {\sc L. C. de Andr\'es, L. A. Cordero, M. Fern\'andez, J. J. Menc\'ia}, Examples of
four dimensional locally conformal K\"ahler solvmanifolds, \textit{ Geom. Dedicata} {\bf 29}
(1989), 227--232.

\bibitem{BD} {\sc M. L. Barberis, I. G. Dotti}, Abelian complex structures on solvable Lie algebras,
\textit{J. Lie Theory} {\bf 14} (2004), 25--34.

\bibitem{BDV} {\sc M. L. Barberis, I. Dotti, M. Verbitsky}, Canonical bundles of complex
nilmanifolds, with applications to hypercomplex geometry, \textit{Math. Research Letters} {\bf
16}(2) (2009), 331--347.

%

\bibitem{CFL} {\sc L. A. Cordero, M. Fern\'andez, M. de L\'eon}, Compact locally conformal K\"ahler
nilmanifolds, \textit {Geom. Dedicata} {\bf 21} (1986), 187--192.

\bibitem{DF} {\sc I. Dotti, A. Fino}, Hyper-K\"ahler torsion structures invariant by nilpotent Lie
groups, \textit{Classical Quantum Gravity} {\bf 19} (2002), 551--562.

\bibitem{DO} {\sc S. Dragomir, L. Ornea}, \textit{Locally conformal K\"ahler geometry}, Progress in
Mathematics v. 155, Birkh\"auser.

\bibitem{GH} {\sc A. Gray, L. Hervella}, \textit{The sixteen classes of almost Hermitian manifolds
and their linear invariants}, Ann. Mat. Pura Appl. (4) {\bf 123} (1980), 35--58.

\bibitem{HK} {\sc K. Hasegawa, Y. Kamishima}, Locally conformally K\"ahler structures on homogeneous
spaces, preprint 2011, arXiv:1101.3693.

\bibitem{Kam} {\sc Y. Kamishima}, Note on locally conformal K\"ahler surfaces,
\textit{Geom. Dedicata} {\bf 84} (2001), 115--124.

\bibitem{KS} {\sc T. Kashiwada, S. Sato}, On harmonic forms in compact locally conformal K\"ahler
manifolds with the parallel Lee form, \textit{Ann. Fac. Sci. Univ. Nat. Za\"ire (Kinshasa) Sect.
Math.-Phys.} {\bf 6} (1980), 17--29.

\bibitem{K} {\sc H. Kasuya}, Vaisman metrics on solvmanifolds and Oeljeklaus-Toma manifolds,
\textit{Bull. London Math. Soc.} {\bf 45} (2013) 15--26.

\bibitem{mpps} {\sc C. Maclaughlin, H. Pedersen, Y.S.Poon, S. Salamon}, Deformation of 2-step
nilmanifolds with abelian complex structures, {\it J. London Math. Soc.} (2) {\bf 73} (2006),
173--193.

\bibitem{Mi} {\sc J. Milnor}, Curvatures of left invariant metrics on Lie groups.
\textit{Adv. Math.} {\bf 21} (1976), 293--329.

\bibitem{L} {\sc H. C. Lee}, A kind of even dimensional differential geometry and its application to
exterior calculus, \textit{American J. Math.} {\bf 65} (1943), 433--438.

\bibitem{L1} {\sc P. Libermann}, Sur le probl\`eme d'equivalence de certaines structures
infinitesimales r\'eguli\`eres, \textit{Annali Mat. Pura Appl.} \textbf{36} (1954), 27--120.

\bibitem{OT} {\sc K. Oeljeklaus, M. Toma}, Non-K\"ahler compact complex manifolds associated to
number fields, \textit{Ann. Inst. Fourier (Grenoble)} {\bf 55} (2005), 161--171.


\bibitem{OV1} {\sc L. Ornea, M. Verbitsky}, Topology of locally conformally K\"ahler manifolds with
potential, \textit{Int. Math. Res. Not.} {\bf 2010} (2010), 717--726.

\bibitem{P} {\sc A. P. Petravchuk}, Lie algebras decomposable into a sum of an abelian and a
nilpotent subalgebra, \textit{Ukr. Math. J.} {\bf 40} (3) (1988), 331--334.

\bibitem{S} {\sc H. Sawai}, Locally conformal K\"ahler structures on compact nilmanifold with
left-invariant complex structures, \textit{Geom. Dedicata} {\bf 125} (2007), 93--101.

\bibitem{S1} {\sc H. Sawai}, Locally conformal K\"ahler structures on compact solvmanifolds,
\textit{Osaka J. Math.} {\bf 49} (2012), 1087--1102.

\bibitem{T} {\sc F. Tricerri}, Some examples of locally conformal K\"ahler manifolds, \textit{Rend.
Sem. Mat. Univ. Politec. Torino} \textbf{40} (1982), 81--92.

\bibitem{U} {\sc L. Ugarte}, Hermitian structures on six-dimensional nilmanifolds,
\textit{Transform. Groups} 12 (1), (2007), 175--202.

\bibitem{V2} {\sc I. Vaisman}, Locally conformal K\"ahler manifolds with parallel Lee form,
\textit{Rend. Mat., VI Ser.} {\bf 12} (1979), 263--284.

\bibitem{V3} {\sc I. Vaisman}, Generalized Hopf manifolds, \textit{Geom. Dedicata} {\bf 13} (1982),
231--255.

\bibitem{V4} {\sc I. Vaisman}, On locally conformal almost K\"ahler manifolds, \textit{Israel J.
Math.} {\bf 24} (1976), 338--351.

\bibitem{W} {\sc H. C. Wang}, Complex parallelisable manifolds, \textit{Proc. Amer. Math. Soc.} {\bf
5} (1954), 771--776.

\end{thebibliography}
\end{document}